\documentclass[fleqn,preprint,3p,a4paper]{elsarticle}

\usepackage{amssymb}
\usepackage{amsmath}
\usepackage{amsthm}
\usepackage{dcolumn}
\usepackage{endnotes}
\usepackage{tabularx}
\usepackage[matrix,arrow]{xy}
\usepackage{wasysym}
\usepackage[colorlinks,
            linkcolor=red,
            anchorcolor=red,
            citecolor=red]{hyperref}

\theoremstyle{plain}

  \newtheorem{thm}{Theorem}[section]
  \newtheorem{lem}[thm]{Lemma}
  \newtheorem{prop}[thm]{Proposition}
  \newtheorem{cor}[thm]{Corollary}
\theoremstyle{definition}
  \newtheorem{defn}[thm]{Definition}
  \newtheorem{exmp}[thm]{Example}

  \newtheorem{rem}[thm]{Remark}

\theoremstyle{break}

\newcommand{\Ua}[0]{\makebox[1ex][l]{\lower.15ex
                                 \hbox{$\uparrow$}}\kern-1ex\lower-.15ex
                                 \hbox{$\uparrow$}}

\newcommand{\Da}[0]{\makebox[1ex][l]{\lower.15ex
                                 \hbox{$\downarrow$}}\kern-1ex\lower-.15ex
                                 \hbox{$\downarrow$}}
\journal{}

\begin{document}

\begin{frontmatter}



\title{One-step Closure, Ideal Convergence and Monotone Determined Space\tnoteref{t1}}
\tnotetext[t1]{This research is supported by the NSFY of China (Nos. 11401435)}.

\author[Wang]{Wu Wang\corref{mycorrespondingauthor}}
\cortext[mycorrespondingauthor]{Corresponding author}
\ead{wangwu@alu.scu.edu.cn}
\address[Wang]{Basic Course Department, Zhonghuan Information College Tianjin University of Technology, Tinjin 300380, China}

\begin{abstract}
Monotone determined spaces are natural topological extensions of dcpo. Its main purpose is to build an extended framework for domain theory. In this paper, we study the one-step closure and ideal convergence on monotone determined space. Then we also introduce the equivalent characterizations of c-spaces and locally hypercompact space. The main results are: (1) Every c-space has one-step closure and every locally hypercompact space has weak one-step closure; (2) A monotone determined space has one-step closure if and only if it is d-meet continuous and has weak one-step closure. (3) $\mathcal{IS}$-convergence(resp. $\mathcal{IGS}$-convergence ) is topological iff $X$ is a c-space (resp. locally hypercompact space); (4) If $X$ is a d-meet continuous space, then the following three conditions are equivalent to each other: (i) $X$ is c-space; (ii) The net $(x_{j})_{j\in J}$ $\mathcal{ISL}$-converges to $x$ iff $(x_{j})_{j\in J}$ $\mathcal{I}$-converges to $x$ with respect to Lawson topology; (iii) The net $(x_{j})_{j\in J}$ $\mathcal{IGSL}$-converges to $x$ iff $(x_{j})_{j\in J}$ $\mathcal{I}$-converges to $x$ with respect to Lawson topology.
\end{abstract}





\end{frontmatter}


\section{Introduction}\label{sec1}

As one of the important branches of mathematics, poset theory aims to provide mathematical models for computer programming languages\cite{1,2,3,6,7}. Continuous domain and quasicontinuous domain are the most important research objects in poset theory. Domain theory is strongly connected with general topology \cite{4,8}. The extension of poset theory to topological space is one of the important research directions of poset theory\cite{8,9,10,11,12}.

Monotone determined spaces were introduced in \cite{16,17,25}, and were shown to be very appropriate topological extensions
of dcpos. Many classical structures such as c-spaces, b-spaces, locally hypercompact spaces are all monotone determined spaces. Monotone determined spaces form a category that has
many fine properties. For example, the category with monotone determined spaces as objects and continuous functions as morphisms is a cartesian closed category.

In \cite{25}, Xie and Kou introduced the approximation relation on $T_{0}$ topological spaces through convergent subsets and then defined d-continuous spaces. They showed that a monotone determined spaces is continuous iff it is a c-space. Zou, Li and Ho introduced the definition of one-step closure on poset\cite{26} and showed that every continuous domain has one-step closure.

The convergence of nets is an important tool in the study of topological spaces. Nets can completely describe the open sets of topological spaces. The notions of liminf convergence and $\mathcal{S}$-convergence were introduced to characterize continuous domains\cite{8}. It was shown that $\mathcal{S}$-convergence on a dcpo $L$ is topological iff $L$ is a continuous dcpo. In \cite{5}, the notion of $\mathcal{S}^{*}$-convergence was introduced to characterize
the notion of quasicontinuous domains. It was shown that $\mathcal{S}^{*}$-convergence is topological iff $L$ is quasicontinuous and the topology generated by it coincides with the Scott topology. In \cite{23}, D. N. Georgiou, A. C. Megaritis, I. Naidoo, et al. introduced another convergence of nets on posets, called ideal-lim-inf-convergence. They showed that for a poset $L$, the ideal-lim-inf-convergence is topological if and only if $X$ is a continuous poset.

In this paper, we shall introduce the definitions of one-step closure and weak one-step closure on directed spaces. And then we define the notions of $\mathcal{IS}$-convergence, $\mathcal{IGS}$-convergence
on $T_{0}$ topological space $X$ as generalizations of ideal-lim-inf-convergence in domain theory and $\mathcal{ISL}$-convergence, $\mathcal{IGSL}$-convergence as generalizations of liminf convergence. At last some characterizations of c-spaces and locally hypercompact spaces will be given.

In Section 3, We first give the equivalent conditions of c-spaces. Then we prove that every c-space has one-step closure and show that monotone determined space is d-meet continuous space if it has one-step closure. In Section 4, we prove that every locally hypercompact spaces has weak one-step closure and show that a monotone determined space has one-step
closure if and only if it is d-meet continuous and has weak one-step closure. In Section 5 and 6,  we draw the following conclusions: (1) $\mathcal{IS}$-convergence(resp. $\mathcal{IGS}$-convergence ) is topological iff $X$ is a
c-space (resp. locally hypercompact space). (2) If $X$ is a d-meet continuous space, then the following three conditions are equivalent
to each other: (i)$X$ is c-space; (ii)The net $(x_{j})_{j\in J}$ $\mathcal{ISL}$-converges to $x$ iff $(x_{j})_{j\in J}$ $\mathcal{I}$-converges to $x$ with respect to Lawson topology; (iii) The net $(x_{j})_{j\in J}$ $\mathcal{IGSL}$-converges to $x$ iff $(x_{j})_{j\in J}$ $\mathcal{I}$-converges to $x$ with respect to Lawson topology.

\section{Preliminaries}\label{sec2}

Now, we introduce the concepts needed in this paper. The readers can also consult\cite{8}. A nonempty set $L$ endowed with a partial order $\leq$ is called  poset. A poset $L$ is called a directed complete poset (dcpo, for short) if any directed subset of $L$ has a sup in $L$. For $x, y\in L$, we say that $x$ is way below $y$ (denoted by $x\ll_{s} y$) if for any directed set $D$ of $L$, $y\leq \vee D$ implies that $x\leq d$ for some $d\in D$. A poset $L$ is called continuous if for any $x\in L$, $\{a\in L: a\ll_{s} x\}$ is directed set and has $x$ as its supremum. For a subset $A$ of $L$, let $\uparrow A =\{x\in L: \exists a\in A, a\leq x\}$, $\downarrow A =\{x\in L:\exists a\in A, x\leq a\}$. We use $\uparrow a$ (resp.$\downarrow a$) instead of $\uparrow \{a\}$(resp. $\downarrow \{a\}$) when $A=\{a\}$. $A$ is called an upper (resp. a lower) set if $A=\uparrow A$ (resp. $A=\downarrow A$).

Let $L$ be a poset and $U\subseteq L$. Then $U$ is called Scott open iff it satisfies: (1) $U=\uparrow U$; (2) For any directed sets $D\subseteq L$, $\vee D\in U$ implies $D\cap U\neq\emptyset$. The collection of all Scott open subsets of $L$ is called the Scott topology and denoted by $\sigma(L)$.

In this paper, topological spaces will always be supposed to be $T_{0}$ spaces. For a topological space $X$, its topology is denoted
by $\tau$. The partial order $\leq$ defined on $X$ by $x\leq y \Leftrightarrow x\in cl_{\tau}\{y\}$ is called the specialization order,
where $cl_{\tau}\{y\}$ is the closure of $\{y\}$. From now on, all order theoretical statements about $T_{0}$ spaces, such as
upper sets, lower sets, directed sets, and so on, always refer to the specialization order. Unless otherwise specified, the internal and closure in this paper are always relative to the topology $\tau$.

A net of a topological space $X$ is a map $\xi: J\rightarrow X$, where $J$ is a directed set. Usually, we denote a net by $(x_{j})_{j\in J}$. Let
$x\in X$, saying $(x_{j})_{j\in J}$ converges to $x$, denote by $(x_{j})_{j\in J}\rightarrow_{\tau} x$, if $(x_{j})_{j\in J}$ is eventually in every
open neighborhood of $X$, that is, for any given open neighborhood $U$ of $x$, there exists $j_{0}\in J$ such that for any $j\in J$, $j\geq j_{0}\Rightarrow x_{j}\in U$.

Let $X$ be a $T_{0}$ topological space,  then any directed subset of $X$ can be regarded as a net, and its index set is itself. We use $D\rightarrow_{\tau} x$ to represent $D$ converges to $x$. Define notation $D(X)=\{(D, x): x\in X, D$ is a directed subset of $X$ and $D\rightarrow_{\tau} x\}$. It is easy to verify that, for any $x, y\in X$, $x\leq y\Leftrightarrow \{y\}\rightarrow_{\tau} x$. Therefore, if $x\leq y$, then
$(\{y\}, x)\in D(X)$. Next, we give the concept of monotone determined space. A subset $U$ of $X$ is called a monotone determined open set if $\forall(D, x)\in D(X)$, $x\in U\Rightarrow D\cap U\neq \emptyset$. Denote all monotone determined  open sets of $X$ by $\mathcal{D}(X)$. Obviously, every open set of X is monotone determined  open, that is, $\tau\subseteq \mathcal{D}(X)$.

\begin{defn}\label{2.1}\cite{16,17,25}
Let $X$ be a $T_{0}$ topological space. $X$ is called monotone determined space if every monotone determined open set of $X$ is an open set, that is, $\mathcal{D}(X)=\tau$.
\end{defn}

\begin{rem}\label{2.2} \cite{16,17,25}
Let $X$ be a $T_{0}$ topological space.

(1) The definition of monotone determined space here is equivalent to definition in \cite{17}.

(2) Every poset equipped with the Scott topology is a monotone determined space \cite{16,18}, besides, each
Alexandroff space is a monotone determined space. Thus, the monotone determined space extends the concept of
the Scott topology.

(3) If $U\in \mathcal{D}(X)$, then $U=\uparrow U$.

(4) $X$ equipped with $\mathcal{D}(X)$ is a $T_{0}$ topological space such that $\leq_{\mathcal{D}}=\leq$, where $\leq_{\mathcal{D}}$ is the specialization
order relative to $\mathcal{D}(X)$.

(5) For a directed subset $D$ of $X$, $D\rightarrow_{\tau} x\Leftrightarrow D\rightarrow_{\mathcal{D}(X)} x$ for any $x\in X$, where $D\rightarrow_{\mathcal{D}(X)} x$ means that $D$ converges to $x$ with respect to the topology $\mathcal{D}(X)$.
\end{rem}

\begin{defn}\label{2.3} \cite{19,25}
Let $X$ be a monotone determined space and $x, y\in X$. We say that $x$ is d-way below $y$, denoted by $x\ll_{d} y$, if for any directed subset $D$ of $X$, $D\rightarrow_{\tau} y$ implies $x\leq d$ for some $d\in D$. If $x\ll_{d} x$, then $x$ is called d-compact element of $X$.
\end{defn}

For any monotone determined space $X$ and $x\in X$, we denote $K_{d}(X)=\{x\in X: x\ll_{d} x\}$, $\Da_{d} x=\{y\in X: y\ll_{d} x\}$,
and $\Ua_{d} x=\{y\in X: x\ll_{d} y\}$. The d-way below relation is a natural extension of way-below relation. Similarly, the notion
of continuous space can be defined.

\begin{defn}\label{2.4} \cite{19,25}
Let $X$ be a monotone determined space. $X$ is called d-continuous if $\Da_{d} x$ is directed and $\Da_{d} x\rightarrow_{\tau} x$ for any $x\in X$.
\end{defn}

A topological space $X$ is called a c-space, if it is $T_{0}$ space and for any $y\in U\in \tau$, there exists $x\in X$ such that $y\in int_{\tau}(\uparrow x)\subseteq \uparrow x\subseteq U$.

Many scholars have done a lot of researches on c-space, and have given many meaningful conclusions about c-space. For example, a poset equipped with Scott topology is c-space; The concept of c-cone is given by combining c-space with topological cone in \cite{20}.

\begin{thm}\label{2.5}\cite{19,25}
Let $X$ be a monotone determined space. Then the following three conditions are equivalent to each other.

(1) $X$ is a d-continuous space;

(2) $X$ is a c-space;

(3) There exists a directed set $D\subseteq \Da_{d} x$ such that $D\rightarrow_{\tau} x$ for any $x\in X$.
\end{thm}

\begin{thm}\label{2.6} \cite{19,25}
Let $X$ be a d-continuous space. Then we have the following statements.

(1) For all $x, y\in X$, $x\ll_{d} y$ implies $x\ll_{d}z\ll_{d} y$ for some $z\in X$.

(2) $\{\Ua_{d} x: x\in X\}$ is a basis of $\tau$.

(3) For all $x, y\in X$, the following are equivalent:

(i) $x\ll_{d} y$;

(ii) $y\in (\uparrow x)^{\circ}$;

(iii) For any net $(x_{j})_{j\in J}\subseteq X$, $(x_{j})_{j\in J}\rightarrow_{\tau} y$ implies $x\leq x_{j_{0}}$ for some $j_{0}\in J$.
\end{thm}

Let $X$ be a $T_{0}$ topological space and $\mathcal{P}^{w}(X)$ be the family of all nonempty finite subsets of $X$. Given any two subsets $G, H$ of $X$, we define $G\leq H$ iff $H\subseteq \uparrow G$. A family of finite sets $\mathcal{F}\subseteq\mathcal{P}^{w}(X)$ is said to be directed if given $F_{1}, F_{2}$ in the family, there exists $F\in \mathcal{F}$ such that $F\subseteq \uparrow F_{1}\cap \uparrow F_{2}$. Let $\mathcal{F}\subseteq \mathcal{P}^{w}(X)$ be a directed family. We say that $\mathcal{F}\rightarrow_{\tau}x$ if for any open neighbourhood $U$ of $x$, there exists some $F\in \mathcal{F}$ such that $F\subseteq U$.

\begin{defn}\label{2.7}\cite{21,27}
 Let $X$ be a monotone determined space and $G, H\subseteq X$. We say that $G$ d-approximates $H$, denoted by $G\ll_{d} H$, if for any directed subset $D$ of $X$, $D\rightarrow_{\tau} h$ for some $h\in H$ implies $D\cap \uparrow G\neq \emptyset$.
\end{defn}

We
write $G\ll_{d} x$ for $G\ll_{d} \{x\}$. $G$ is said to be compact if $G\ll_{d} G$. For any $T_{0}$ topological space $X$ and $F\subseteq X$, we denote $\Uparrow_{d} F=\{x\in X: F\ll_{d} x\}$.

\begin{defn}\label{2.8} \cite{21,27} A topological space $X$ is called d-quasicontinuous if it is a monotone determined space such that
for any $x\in X$, the family $fin_{d}(x)=\{F: F\in \mathcal{P}^{w}(X), F\ll_{d} x\}$ is a directed family and converges to $x$.
\end{defn}

\begin{thm}\label{2.9} \cite{21,27} Let $X$ be a
d-quasicontinuous space. The following statements hold.

(1) Given any $H\in \mathcal{P}^{w}(X)$ and $y\in X$, $H\ll_{d}y$ implies $H\ll_{d}F\ll_{d}y$ for some finite subset $F\in \mathcal{P}^{w}(X)$.

(2) Given any $F\in \mathcal{P}^{w}(X)$, $\Uparrow_{d} F =(\uparrow F)^{\circ}$. Moreover, $\{\Uparrow_{d} F : F\in \mathcal{P}^{w}(X)\}$ is a base of $\tau$.
\end{thm}

\begin{defn}\label{2.10} \cite{27} A $T_{0}$ topological space $X$ is called locally hypercompact, if for any open subsets $U$ of $X$ and $x\in U$, there exists some $F\in \mathcal{P}^{w}(X)$ such that $x\in (\uparrow F)^{\circ} \subseteq \uparrow F\subseteq U$.
\end{defn}

\begin{thm}\label{2.11} \cite{21,27}
 Let $X$ be a monotone determined space, Then the following three conditions are equivalent to each other.

(1) $X$ is a d-quasicontinuous space;

(2) $X$ is a locally hypercompact space;

(3) There exists a directed family $\mathcal{F }\subseteq fin(x)$ such that $\mathcal{F}\rightarrow_{\tau} x$ for any $x\in X$.
\end{thm}

\begin{lem}\label{2.12} \cite{21}
 Let $X$ be a $T_{0}$ topological space, $\mathcal{F}\subseteq \mathcal{P}^{w}(X)$ be directed family, $G, H\in \mathcal{P}^{w}(X)$, $x\in X$. If $G\ll _{d} H$ and $\mathcal{F}\rightarrow_{\tau} x\in H$, then there exists $F\in \mathcal{F}$ such that $F\subseteq \uparrow G$.
\end{lem}

Continuous dcpos and quasicontinuous dcpos endowed with the Scott topology can be viewed as special d-continuous spaces and d-quasicontinuous spaces.

\begin{defn}\label{2.13} \cite{21}
 A topological space $X$ is called d-meet continuous if it is a monotone determined space such that
for any $x\in X$ and directed subset $D$ of $X$, $D\rightarrow_{\tau}x$ implies $x\in cl(\downarrow D\cap \downarrow x)$.
\end{defn}

\begin{lem} \label{2.14}
Let $X$ be a monotone determined space, then $X$ is d-meet continuous iff for any $F\in \mathcal{P}^{w}(X)$, $U\in \tau$, $\uparrow (U\cap \downarrow F)\in \tau$.
\end{lem}

\begin{proof} If Let $X$ be monotone determined space. By Theorem 4.6 in \cite{21}, $X$ is d-meet continuous iff $\uparrow (U\cap \downarrow x)\in \tau$ for any $F\in \mathcal{P}^{w}(X)$, $U\in \tau$.
For any $F\in \mathcal{P}^{w}(X)$, $\uparrow (U\cap \downarrow F)=\uparrow (U\cap \bigcup_{f\in F}\downarrow f)=\uparrow\bigcup_{f\in F}(U\cap \downarrow x)=\bigcup_{f\in F}(\uparrow U\cap \downarrow x)\in \tau$.

Conversely, let $F=\{x\}$, then $\uparrow (U\cap \downarrow x)\in \tau$.
\end{proof}
A family $\mathcal{I}$ of subsets of a non-empty set $M$ is called an ideal if $\mathcal{I}$ has the following properties: (1) If $A\in \mathcal{I}$ and $B\subseteq A$, then $B\in \mathcal{I}$; (2) If $A, B\in \mathcal{I}$, then $A\cup B\in \mathcal{I}$. The ideal $\mathcal{I}$ is called non-trivial or proper if $M\notin \mathcal{I}$.

Let $D$ be a directed set. For all $d\in D$ we set $M_{d}= \{d'\in D: d\leq d'\}$. An ideal $\mathcal{I}$ of a directed set $D$ is
called D-admissible, if $D\backslash M_{d}\in \mathcal{I}$, for any $d\in D$. Then, $\mathcal{I}_{0}=\mathcal{I}_{0}(D)=\{A\subseteq D : A\subseteq D\backslash M_{d}$ for some $d\in D\}$ is a non-trivial D-admissible ideal of $D$\cite{22,23}. Let $F\in \mathcal{P}^{w}(D)$ and $M_{F}= \{d'\in D: d'\in \uparrow F\}$, then $M_{f}\subseteq M_{F}$ hold for any $f\in F$ and $D\backslash M_{F}\subseteq D\backslash M_{f}\in \mathcal{I}_{0}$, i.e., $D\backslash M_{F}\in \mathcal{I}_{0}$.

\begin{lem} \label{2.15}
Let $X$ be a $T_{0}$ topological space, $(x_{j})_{j\in J}\subseteq X$ be a net and $A\subseteq X$. If $\{j\in J: x_{j}\notin \uparrow A\}\in \mathcal{I}_{0}$, then $(x_{j})_{j\in J}$ is eventually in $\uparrow A$.
\end{lem}

\begin{proof}If $\{j\in J: x_{j}\notin \uparrow A\}\in \mathcal{I}_{0}(J)$, then $\{j\in J: x_{j}\notin \uparrow A\}\subseteq J\backslash M_{j_{0}}$ for some $j_{0}\in J$.  So, $M_{j_{0}}\subseteq \{j\in J: x_{j}\in \uparrow A\}$. Therefore, if $j\geq j_{0}$, then $x_{j}\in \uparrow A$, that is, $(x_{j})_{j\in J}$ is eventually in $\uparrow A$.
\end{proof}

\begin{defn}\label{2.16} \cite{22,23}
Let $X$ be a topological space. A net $(x_{j})_{j\in J}\subseteq X$ is said to $\mathcal{I}$-converge
to a point $x\in X$, where $\mathcal{I}$ is an ideal of $J$, if for every open neighborhood $U$ of $x$, $\{j\in J: x_{j}\notin U\}\in \mathcal{I}$. In this case the point $x$ is called the $\mathcal{I}$-limit of the net $(x_{j})_{j\in J}$ and we write $(x_{j})_{j\in J}\rightarrow_{\mathcal{I}}x$.
\end{defn}

Let X be a topological space, $(x_{j})_{j\in J}\subseteq X$ be a net and $x\in X$. Then $(x_{j})_{j\in J}\rightarrow _{\tau}x$ iff $(x_{j})_{j\in J}\rightarrow_{\mathcal{I}_{0}}x$ \cite{24}.

\begin{defn}\label{2.17} \cite{16,25} Suppose $X, Y$ are two $T_{0}$ topological spaces. A function $f : X\rightarrow Y$ is called monotone determined continuous if it is monotone and preserves all limits of directed subset of $X$; that is, $(D, x)\in D(X)\Rightarrow (f(D), f(x))\in D(Y)$.
\end{defn}
Here are some characterizations of the directed continuous functions.

\begin{prop}\label{2.18} \cite{16,25}  Suppose $X, Y$ are two $T_{0}$ topological spaces. $f : X\rightarrow Y$ is a function between $X$
and $Y$. Then

(1) $f$ is monotone determined continuous if and only if $\forall U \in d(Y)$, $f^{-1}(U)\in d(X)$.

(2) If $X, Y$ are monotone determined spaces, then $f$ is continuous if and only if it is monotone determined continuous.
\end{prop}
\section{One-step closure and c-space}\label{sec3}

In the previous section, we see that a monotone determined space is c-space iff it is d-continuous space. In this section, We first give the equivalent conditions of c-spaces. And then we introduce the concept
of one-step closure and show that every continuous space has one-step closure. At the same time, we prove that all monotone determined spaces having one-step closure are d-meet continuous space and continuous retract of monotone determined space which has one-step closure has one-step closure.
\begin{defn}\label{3.1}
Let $X$ be a monotone determined space and $A\subseteq L$,  the lower and upper approximations of $A$, denoted by $A^{\downarrow \ll}$ and $A^{\uparrow \ll}$ respectively, are defined by $A^{\downarrow \ll}=\{x\in X: x\in A$ and $\Da_{d} x\cap A\neq \emptyset\}$ and
$A^{\uparrow \ll}=\{x\in X: \Da_{d} x\subseteq \downarrow A\}$.
\end{defn}

Let $X$ be a monotone determined space and $a\in X$, then the followings are immediate consequences arising from Definition \ref{3.1}: (1) $(\uparrow a)^{\downarrow\ll}=\Ua_{d} a$; (2) For any subset $A\subseteq X$ and $x\in A$, $x\in A^{\downarrow \ll}$ iff  there exists $y\in A$ such that $y\ll_{d} x$; (3) $A^{\uparrow \ll}$ is lower set; (4) If $A$ is a upper set, $A^{\downarrow \ll}$ is upper set.

\begin{prop}\label{3.2}
Let $X$ be a monotone determined space and $U\subseteq X$. If $U$ is upper set and $U=U^{\downarrow \ll}$, then $U$ is
monotone determined open.
\end{prop}
\begin{proof}
Since $U$ is upper, it suffices to show that for any directed $D\subseteq X$ and $D\rightarrow_{\tau} x\in U$, $D\cap U\neq \emptyset$. To this end, let $D$ be a directed subset and $D\rightarrow_{\tau} x\in U$. Because $U=U^{\downarrow \ll}$, $\Da_{d} x\cap U\neq\emptyset$. Hence there exists $u\in U$ such that already $u\ll_{d} x$. Since $D\rightarrow_{\tau} x$ and $u\ll_{d} x$, there
is a $d\in D$ such that $u\leq d$. Since $U$ is upper and $u\in U$, $d\in U$. Then $D\cap U\neq\emptyset$. That is $U$ is monotone determined open.
\end{proof}

For any topological space $X$ and $A\subseteq X$, it holds that $X\backslash cl(A)=int (X\backslash A)$ and hence $cl(A)\cup int(X\backslash A)=X$ and  $cl(A)\cap int(X\backslash A)=\emptyset$.

\begin{prop}\label{3.3}
 Let $X$ be a directed space. The followings are equivalent.

(1) $X$ is c-space.

(2) $int(A)=A^{\downarrow\ll}$ for any upper set $A\subseteq X$.

(3) $cl(B)=B^{\uparrow\ll}$ for any lower set $A\subseteq X$.
\end{prop}
\begin{proof}
(1)$\Rightarrow$ (2) If $x\in int(A)$, then $x\in A$. Since $X$ is c-space, $\{\Ua_{d} x:x\in X\}$ is a basis of $\tau$ by Theorem \ref{2.5} and \ref{2.6}. Then $x\in \Ua_{d} a\subseteq \uparrow a\subseteq int(A)$ for some $a\in X$. That is $a\in \Da_{d} x\cap A$, then $x\in A^{\downarrow\ll}$ and $int(A)\subseteq A^{\downarrow\ll}$. Let $x\in A^{\downarrow\ll}$. then $\Da_{d} x\cap A\neq \emptyset$. Thus there exists a $a\in \Da_{d} x\cap A$ such that $x\in \Ua_{d} a\subseteq A$. Since $\Ua_{d} a$ is open set, $x\in int(A)$ and $A^{\downarrow\ll}\subseteq int(A)$. Thus we have $int(A)=A^{\downarrow\ll}$.

(2)$\Rightarrow$ (1)
Let $a,b\in \Da_{d} x$, then $x\in \Ua_{d} a\cap \Ua_{d} b$. Since $\Ua_{d} a=(\uparrow a)^{\downarrow\ll}=int(\uparrow a)$ and $\Ua_{d} b=(\uparrow b)^{\downarrow\ll}=int(\uparrow b)$, $x\in \Ua_{d} a\cap \Ua_{d} b=int(\uparrow a)\cap int(\uparrow b)=int(\uparrow a\cap \uparrow b)=(\uparrow a\cap \uparrow b)^{\downarrow \ll}$. Thus $\Da_{d} x\cap (\uparrow a\cap \uparrow b)\neq \emptyset$. Let $c\in\uparrow a\cap \uparrow b$ such that $c\ll_{d} x$. This proves that $\Da_{d} x$ is a directed set.

Let $x\in U\in \mathcal{D}(X)$, then $U=int(U)=U^{\downarrow\ll}$. Since $x\in U^{\downarrow\ll}$, $\Da_{d} x\cap U\neq \emptyset$. Thus $\Da_{d} x\rightarrow_{\tau} x$. Then $X$ is d-continuous by Definition \ref{2.4}. Hence $X$ is a c-space.

(2)$\Leftrightarrow$ (3) Let $A$ is upper set, then $x\in A^{\downarrow \ll}\Leftrightarrow x\in A$ and $\Da_{d} x\cap A\neq \emptyset \Leftrightarrow x\in A$ and $\Da_{d} x\nsubseteq X\backslash A$. Since $X\backslash A=\downarrow X\backslash A$, then $x\in A$ and $x\notin (X\backslash A)^{\uparrow \ll}$. Then $x\in A\cap X\backslash (X\backslash A)^{\uparrow \ll}$ and $A^{\downarrow \ll}\subseteq A\cap X\backslash (X\backslash A)^{\uparrow \ll}$. Conversely, if $x\in A\cap X\backslash (X\backslash A)^{\uparrow \ll}$, then $x\notin (X\backslash A)^{\uparrow \ll}$. Thus there exists $a\in \Da_{d} x$ such that $a\notin X\backslash A$. Hence $a\in A\cap \Da_{d} x$, $x\in A^{\downarrow \ll}$. That is $A\cap X\backslash (X\backslash A)^{\uparrow \ll}\subseteq A^{\downarrow \ll}$. Then $A\cap X\backslash (X\backslash A)^{\uparrow \ll}= A^{\downarrow \ll}$. Since $X\backslash A\subseteq (X\backslash A)^{\uparrow \ll}$, $A^{\downarrow \ll}\cup (X\backslash A)^{\uparrow \ll}=X$ and $A^{\downarrow \ll}\cap (X\backslash A)^{\uparrow \ll}=\emptyset$.

Now we show that (2)$\Leftrightarrow$ (3). If $B\subseteq X$ is a lower set, then $cl(B)=X\backslash int(X\backslash B)=X\backslash (X\backslash B)^{\downarrow \ll}$. Let $A= X\backslash B$, then $cl(B)=X\backslash (A)^{\downarrow \ll}=(X\backslash A)^{\uparrow \ll}=B^{\uparrow \ll}$. Conversely, if $A\subseteq X$ is a upper set, then $int(A)=X\backslash cl(X\backslash A)=X\backslash (X\backslash A)^{\uparrow \ll}=A^{\downarrow \ll}$.
\end{proof}

\begin{cor}\label{3.4}
 Let $X$ be a  c-space, $A\subseteq X$ and $x\in X$. Then $x\in cl(A)$  if and only if $\Da_{d} x\subseteq \downarrow A$.
\end{cor}
\begin{proof}
By Proposition \ref{3.3}, $cl(A)=cl(\downarrow A)=(\downarrow A)^{\uparrow\ll}=A^{\uparrow\ll}$. So $x\in cl(A)$  if and only if $\Da_{d} x\subseteq \downarrow A$.
\end{proof}

Let $X$ be a  monotone determined space and $A\subseteq X$, we define $\widetilde{A}=\{x\in X: \exists$ directed set $D\subseteq \downarrow A\cap\downarrow x$, $D\rightarrow_{\tau} x\}$.

\begin{prop}\label{3.5}
 Let $X$ be a  monotone determined space and $A\subseteq X$. The following hold.

(1) $A\subseteq \downarrow A\subseteq \widetilde{A}\subseteq cl(A)$.

(2) $\widetilde{A}\subseteq A^{\uparrow \ll}$.
\end{prop}
\begin{proof}
(1) Obvious.

(2) Suppose $x\in \widetilde{A}$. Then there is a directed subset $D$ of $\downarrow A\cap \downarrow x$ such that $D\rightarrow_{\tau} x$. For each
$y\ll_{d} x$, there is an element $d\in D$ such that $y\leq d$. Since $d\in \downarrow A$ and $y\leq d$, $y\in \downarrow A$.  So that $\Da_{d} x\subseteq \downarrow A$.
Hence $x\in A^{\uparrow \ll}$.
\end{proof}

By Proposition \ref{3.3} and \ref{3.5}, if $X$ is d-continuous space and $A\subseteq X$ is a lower set, then $\widetilde{A}\subseteq A^{\uparrow \ll}=cl(A)$.

\begin{defn}\label{3.6}
 A  monotone determined space $X$ is said to have one-step closure if $\widetilde{A}=cl(A)$ for every subset $A\subseteq X$.
\end{defn}

\begin{prop}\label{3.7}
 Every c-space has one-step closure.
\end{prop}
\begin{proof}
Let $X$ be a c-space, then $X$ is a d-continuous poset If $A\subseteq X$, then, on one hand, $\widetilde{A}\subseteq cl(A)$ holds
by Proposition \ref{3.5}. On the other hand, if $x\in cl(A)$, then $\Da_{d} x\subseteq \downarrow A$ by virtue of
Corollary \ref{3.4}. But $X$ is continuous so that $\Da_{d} x$ is directed and $\Da_{d} x\rightarrow x$. Since directed set $\Da_{d} x\subseteq \downarrow A\cap \downarrow x$ and $\Da_{d} x\rightarrow x$, $x\in \widetilde{A}$. Thus, we can conclude that $\widetilde{A}=cl(A)$.
\end{proof}

\begin{prop}\label{3.8}
All  monotone determined spaces having one-step closure are d-meet continuous space.
\end{prop}
\begin{proof}
Suppose that there is a directed set $D\subseteq X$ and $D\rightarrow_{\tau} x$. Since $X$ has one-step closure, $x\in cl(D)=\widetilde{D}$. So, there exists a directed set $D'\subseteq \downarrow D \cap \downarrow x$ such that $D'\rightarrow_{\tau} x$. Since $D'\subseteq (\downarrow D \cap \downarrow x)\cap \downarrow x$ and $D'\rightarrow_{\tau} x$, $x\in \widetilde{\downarrow D \cap \downarrow x}$. Since $X$ has one-step closure, $\widetilde{\downarrow D \cap \downarrow x}=cl(\downarrow D \cap \downarrow x)$ and $x\in cl(\downarrow D \cap \downarrow x)$. This completes the proof that $X$ is d-meet continuous.
\end{proof}

\begin{prop}\label{3.9}
Let $X$ be a  monotone determined space that has one-step closure. Then $int(\uparrow x)=\Ua_{d} x$ for each $x\in X$.

\end{prop}
\begin{proof}
Since $int(\uparrow x)\subseteq \Ua_{d} x$ always holds for any element $x\in X$, it remains to prove the reverse
containment. For the sake of contradiction, suppose that there is $y\in \Ua_{d} x\backslash int(\uparrow x)$. Since $X$ be a  monotone determined space that has one-step closure, $cl(X\backslash \uparrow x) = \widetilde{X\backslash \uparrow x}$. Then
$y\in X\backslash int(\uparrow x)= cl(X\backslash \uparrow x) = \widetilde{X\backslash \uparrow x}$. So there is a directed subset $D\subseteq \downarrow (X\backslash \uparrow x)\cap\downarrow y= (X\backslash \uparrow x)\cap \downarrow y$
such that $D\rightarrow y$ and $D\cap \uparrow x=\emptyset$. This will run contrary to $x\ll_{d} y$. Hence $\Ua_{d} x\subseteq int(\uparrow x)$.
\end{proof}

Let $X, Y$ be two topological space. $Y$ is called a continuous retract of $X$ if
there exist continuous functions $r: X\rightarrow Y$ and $s :Y\rightarrow X$ such that $r\circ s = 1_{Y}$.

\begin{prop}\label{3.10}
Let $X, Y$ be a  monotone determined space and $Y$ be a continuous retract of $X$. If $X$ has  one-step closure, then $Y$ has one-step closure.
\end{prop}
\begin{proof}
Since $Y$ is a continuous retract of $X$, we have that there exists two continuous maps $s: Y\rightarrow X$
and $r : X\rightarrow Y$ such that $id_{Y} = r\circ s$. Let $y\in Y$, $A\subseteq Y$ and $y\in cl_{Y}(A)$, then $s(y)\in s(cl_{Y}(A))\subseteq cl_{X}(s(A))$. Since  $X$ is a  monotone determined space that has one-step closure, $cl_{X}(s(A))=\widetilde{s(A)}$.
We know that there exists $D\subseteq \downarrow s(A)\cap \downarrow s(y)$ such that $D\rightarrow s(y)$. Since $r$ is continuous, $r(D)\rightarrow y$. Since $r(D)\subseteq r(\downarrow s(A)\cap \downarrow s(y))\subseteq r(\downarrow s(A))\cap r(\downarrow s(y))=\downarrow A\cap \downarrow y$, $y\in \widetilde{A}$. That is $cl_{Y}(A)\subseteq\widetilde{A}$. By Proposition \ref{3.5}, $cl_{Y}(A)=\widetilde{A}$ and $Y$ has one-step closure.
\end{proof}

\section{Weak one-step closure}\label{sec4}

In this section, we introduce the concept of weak one-step closure and show that every quasicontinuous space has weak one-step closure. Then we investigate the relationship between weak one-step closure and one-step closure.
\begin{defn}\label{4.1}
Let $X$ be a monotone determined space and $A\subseteq X$, we define $\widehat{A}=\{x\in X: \exists$ directed set $D\subseteq \downarrow A$, $D\rightarrow_{\tau} x\}$.
A monotone determined space $X$ is said to have weak one-step closure if $\widehat{A}=cl(A)$ for every subset $A\subseteq X$.
\end{defn}
Let $X$ be a monotone determined space and $A\subseteq X$. If $x\in \widehat{A}$ and $a\leq x$, there exists $D\subseteq \downarrow A$ such that $D\rightarrow_{\tau} x$. For any $U\in \mathcal{D}(X)$ and $a\in U$, $x\in U$. Then $D\cap U\neq \emptyset$, $D\rightarrow_{\tau} a$. Thus $a\in \widehat{A}$ by Definition \ref{4.1}. This show that $\widehat{A}$ is lower set.

\begin{lem}\cite{2}\label{4.2}(Rudin's Lemma)Let $\mathcal{F}$ be a directed family (with respective to the Smyth preorder) of nonempty finite subsets of a poset $P$. Then there exist a directed set $D\subseteq \bigcup_{F\in \mathcal{F}}F$ such that $D\cap F\neq \emptyset$ for all $F\in \mathcal{F}$.
\end{lem}
\begin{prop}\label{4.3}
 Every locally
hypercompact space has weak one-step closure.
\end{prop}
\begin{proof}
It suffices to show that $cl(A)\subseteq \widehat{A}$ for any subset $A$ of $X$. To this end, let $x\in cl(A)$, $F\in fin_{d}(X)$ with $x\in \Uparrow_{d} F$. Then $\Uparrow_{d} F$ is directed open as $X$ is d-quasicontinuous space by Theorem \ref{2.9} and \ref{2.11}. If $\Uparrow_{d} F\cap A= \emptyset$, then $A\subseteq X\backslash \Uparrow_{d} F$. Thus $cl(A)\subseteq X\backslash \Uparrow_{d} F$, $x\in X\backslash \Uparrow_{d} F$. This will run
contrary to $x\in \Uparrow_{d} F$. Hence $\Uparrow_{d} F\cap A\neq\emptyset$, which implies that
$F\cap \downarrow A\neq \emptyset$. Since $X$ is d-quasicontinuous space, $(F\cap \downarrow A)_{F\in fin_{d}(x)}$
is a directed family (with respective to the Smyth preorder) of nonempty
finite subsets of $X$. By Rudin's Lemma, there exists a directed subset $D$ of $\bigcup_{F\in fin_{d}(x)}F\cap \downarrow A$ such that
$D\cap (F\cap\downarrow A)\neq\emptyset$ for any $F\in fin_{d}(x)$. Also, since $X$ is a d-quasicontinuous space, $fin_{d}(x)\rightarrow_{\tau} x$. That is $D\rightarrow_{\tau} x$ and $x\in \widehat{A}$.
\end{proof}

\begin{prop}\label{4.4}
If a monotone determined space $X$ has one-step closure, then it has weak one-step closure.
\end{prop}
\begin{proof}
To prove the above results, we only need to prove that $cl(A)=\widehat{A}$ for any subset $A$ of $X$. From the definition of one-step closure, we have $cl(A)=\widetilde{A}$. One sees obviously that $\widehat{A}\subseteq cl(A)$. Let $x\in cl(A)$. Then $x\in \widetilde{A}$. It follows that
there exists $D\subseteq \downarrow A\cap \downarrow x$ such that $D\rightarrow_{\tau} x$, i.e., $x \in \widehat{A}$. \end{proof}

\begin{prop}\label{4.5}
Let $X$ be a monotone determined space. Then the following statements are equivalent:

(1) $\widetilde{A}$ is a lower set for any $A\subseteq X$;

(2) $\widetilde{D}$ is a lower set for any directed set $D\subseteq X$;
\end{prop}
\begin{proof}
(1)$\Rightarrow $(2) is trivial.

(2)$\Rightarrow $(1) Assume $x \leq y\in \widetilde{A}$. Then there exists $D\subseteq \downarrow A\cap \downarrow y$ such that $D\rightarrow_{\tau} y$. That is $D\subseteq \downarrow D\cap \downarrow y$ such that $D\rightarrow_{\tau} y$. This means that $y\in \widetilde{D}$. Since $x\leq y$, $x\in \downarrow \widetilde{D}=\widetilde{D}$. So we have that there exists a directed subset $D'$ of $\downarrow D\cap \downarrow x$ such
that $D'\rightarrow_{\tau} x$. Note that $D'\subseteq \downarrow D\cap \downarrow x\subseteq \downarrow A\cap\downarrow x$. Therefore, $x\in \widetilde{A}$.
\end{proof}

\begin{prop}\label{4.6}
Let $X$ be a d-meet continuous space. If $X$ has weak one-step closure, then $X$ has one-step closure.
\end{prop}
\begin{proof}
If $X$ has weak one-step closure, then $\widehat{A}=cl(A)$. Let $x\in cl(A)=\widehat{A}$. Thus there exists directed set $D\subseteq \downarrow A$, $D\rightarrow_{\tau} x$. Since $X$ is d-meet continuous space, $x\in cl(\downarrow D \cap\downarrow x)$. That is $x\in\widehat{\downarrow D \cap\downarrow x}$, then there exist $D'\subseteq \downarrow D \cap\downarrow x$ such that $D'\rightarrow_{\tau} x$. Since $D'\subseteq \downarrow D \cap\downarrow x\subseteq\downarrow A \cap\downarrow x$, $x\in \widetilde{A}$. That show $X$ has one-step closure.
\end{proof}

The main result of this section is given below.
\begin{thm}\label{4.7}
Let $X$ be a monotone determined space. Then $X$ is d-meet continuous space and has one-step closure if and only if $X$ has one-step closure.
\end{thm}
\begin{proof}
($\Rightarrow$) It follows immediately from Proposition \ref{4.6}.

($\Leftarrow$) If $X$ has one-step closure, then $X$ is d-meet continuous space and has one-step closure by Proposition \ref{3.8} and \ref{4.4}
\end{proof}

\section{Ideal $\mathcal{S}$-convergence and c-space}\label{sec5}

It is known that $\mathcal{S}$-convergence in a dcpo is topological if and only if the dcpo is a continuous domain\cite{8}. In \cite{23}, the notion of ideal-lim-inf-convergence was introduced and it was shown that ideal-lim-inf-convergence in a
poset is topological if and only if the poset is a continuous poset.

In this section, We will define the notion of ideal $\mathcal{S}$-convergence on a topological spaces $X$. We show that ideal $\mathcal{S}$-convergence
is topological iff $X$ is a c-space. On the other hand, we gain the convergence theoretical characterizations of d-continuous
spaces. These theorems for ideal-lim-inf-convergence \cite{23} can be viewed as special
cases of ideal $\mathcal{S}$-convergence.

Firstly, we recall the following definition.

\begin{defn}\label{5.1} \cite{23}
A net $(x_{j})_{j\in J}$ in a poset $L$ is said to $\mathcal{I}$-lim-inf-converge to an element $x\in L$, where $\mathcal{I}$ is an ideal of $J$, if there exists a directed subset $D$ of $L$ such that the following conditions hold: (1) $\vee D$ exists with $x\leq \vee D$, and (2) for any $d\in D$, $\{j\in J: x_{j}\ngeq d\}\in \mathcal{I}$. In this case the point $x$ is called the $\mathcal{I}$-lim-inf-limit of the net $(x_{j})_{j\in J}$.
\end{defn}

For any set $X$, we denote $\Phi(X)$ to be the class of $((x_{j})_{j\in J}, \mathcal{I})$ in $X$, where $(x_{j})_{j\in J}$ is a net and $\mathcal{I}$ is a ideal of $J$. A ideal convergence class $\mathcal{L}$ in $X$ is a relation between $\Phi(X)$ and $X$, i.e., $\mathcal{L}$ is a subclass of $\{((x_{j})_{j\in J}, \mathcal{I}, x): (x_{j})_{j\in J}$ is a net in $X$, $\mathcal{I}$ is a ideal of $J$ and $x\in X\}$. We say that a convergence class $\mathcal{L}$ in a set $X$ is topological if there is a topology $\tau$ on $X$ such that  $((x_{j})_{j\in J}, \mathcal{I}, x)\in \mathcal{L}\Leftrightarrow(x_{j})_{j\in J}\rightarrow _{\mathcal{I}}x$.

\begin{thm}\label{5.2} \cite{23}
For a poset $L$, the ideal-lim-inf-convergence is topological if and only if $L$ is a continuous poset.
\end{thm}

\begin{defn}\label{5.3}
 Let $X$ be a $T_{0}$ topological space. A net $(x_{j})_{j\in J}\subseteq X$ is said to $\mathcal{IS}$-converge
to a point $x\in X$, where $\mathcal{I}$ is an ideal of $J$, if there exists a directed subset $D$ of $X$ such that $D\rightarrow_{\tau} x$ and for any $d\in D$, $\{j\in J: x_{j}\ngeq d\}\in \mathcal{I}$. In this case the point $x$ is called the $\mathcal{IS}$-limit of the net $(x_{j})_{j\in J}$ and we write $(x_{j})_{j\in J}\rightarrow _{\mathcal{IS}}x$.
\end{defn}

It is easy to verify that for any $x, y\in X$, $x\leq y\Leftrightarrow \{y\}\rightarrow_{\mathcal{IS}} x$ for trivial ideal.

\begin{lem}\label{5.4}
 Let $X$ be a $T_{0}$ topological space, $x\in X$ and $(x_{j})_{j\in J}\subseteq X$ be a net. Then $(x_{j})_{j\in J}\rightarrow _{\mathcal{IS}}x\Rightarrow (x_{j})_{j\in J} \rightarrow _{\mathcal{I}}x$ with respect to $\tau$. In particular, for any directed subset $D$ of $X$, $D\rightarrow _{\mathcal{I}_{0}\mathcal{S}}x\Leftrightarrow D \rightarrow _{\tau}x$.
\end{lem}

\begin{proof} If $(x_{j})_{j\in J}\rightarrow _{\mathcal{IS}}x$, then there exists a directed subset $D$ of $X$ such that $D\rightarrow_{\tau} x$ and for any $d\in D$, $\{j\in J: x_{j}\ngeq d\}\in \mathcal{I}$. Let $x\in U\in\tau$, then $U\in \mathcal{D}(x)$ and $D\cap U\neq \emptyset$. Let $d\in D\cap U$, then $\{j\in J: x_{j}\notin U\}\subseteq \{j\in J: x_{j}\ngeq d\}\in \mathcal{I}$, that is, $\{j\in J: x_{j}\notin U\}\in \mathcal{I}$. Thus $(x_{j})_{j\in J}\rightarrow _{\mathcal{I}}x$ with respect to $\tau$ by Definition \ref{2.16}.

In particular, $D\rightarrow _{\mathcal{I}_{0}\mathcal{S}}x\Leftrightarrow D \rightarrow _{\tau}x$ be definition of $\mathcal{I}_{0}$.
\end{proof}
\begin{prop}\label{5.5}
  Let $X$ be a monotone determined space and $x, y\in X$. Then $x\ll_{d} y$ if and only if for any net $(x_{j})_{j\in J}$ in $X$ and any non-trivial ideal $\mathcal{I}$ of $J$ such that $(x_{j})_{j\in J}\rightarrow _{\mathcal{IS}}y$, we have $\{j\in J: x_{j}\ngeq x\}\in \mathcal{I}$.
\end{prop}

\begin{proof} Let $x\ll_{d} y$, $(x_{j})_{j\in J}$ be a net in $X$ and $\mathcal{I}$ be a non-trivial ideal of $J$ such that $(x_{j})_{j\in J}\rightarrow _{\mathcal{IS}}x$. Then there exists a directed subset $D$ of $X$ such that $D\rightarrow_{\tau}x$ and for any $d\in D$, $\{j\in J: x_{j}\ngeq d\}\in \mathcal{I}$. Since $x\ll_{d} y$, there exists $d_{0}\in D$ such that $x\leq d_{0}$. Then $\{j\in J: x_{j}\ngeq x\}\subseteq\{j\in J: x_{j}\ngeq d_{0}\}\in \mathcal{I}$. Hence, $\{j\in J: x_{j}\ngeq x\}\in \mathcal{I}$.

Conversely, suppose that for any net $(x_{j})_{j\in J}$ in $X$ and any non-trivial ideal $\mathcal{I}$ of $J$ such that $(x_{j})_{j\in J}\rightarrow _{\mathcal{IS}}x$, we have $\{j\in J: x_{j}\ngeq x\}\in \mathcal{I}$. Let $D$ be a directed subset of $X$ such that $D\rightarrow_{\tau} y$. Consider the net $(x_{d})_{d\in D}$, where $d_{d}=d$, for any $d\in D$. Then $(x_{d})_{d\in D}\rightarrow_{\mathcal{I}_{0}\mathcal{S}}y$ and therefore, by hypothesis, $\{d\in D: x_{d}\ngeq x\}\in \mathcal{I}_{0}$. Since $\mathcal{I}_{0}$ is non-trivial, $\{d\in D: x_{d}\ngeq x\}\neq D$. Hence, there exists $x_{d}=d\in D$ such that $d\geq x$. Thus, $x\ll_{d} y$ by Definition \ref{2.3}.
\end{proof}

\begin{prop} \label{5.6} Let $X$ be a c-space, $x, y\in X$, $(x_{j})_{j\in J}\subseteq X$ be a net. Then $(x_{j})_{j\in J}\rightarrow _{\mathcal{IS}}y$, where $\mathcal{I}$ is a non-trivial ideal of $J$, if and only if for any $x\ll_{d} y$, we have $\{j\in J: x_{j}\ngeq x\}\in \mathcal{I}$.
\end{prop}
\begin{proof} Suppose that $(x_{j})_{j\in J}\rightarrow _{\mathcal{IS}}y$ and $x\ll_{d} y$. Then, by Proposition \ref{5.5}, we have $\{j\in J: x_{j}\ngeq x\}\in \mathcal{I}$.

Conversely, let $(x_{j})_{j\in J}$ be a net in $X$ and for any $x\ll_{d} y$ we have $\{j\in J: x_{j}\ngeq x\}\in \mathcal{I}$, where $\mathcal{I}$ is a non-trivial ideal of $J$. The set $\Da_{d} y$ is directed and $\Da_{d} y\rightarrow_{\tau}y$. By hypothesis and Definition \ref{5.3}, it follows that $(x_{j})_{j\in J}\rightarrow _{\mathcal{IS}}y$.
\end{proof}

We define a convergence class $\mathcal{L}_{\mathcal{IS}}$ on $X$ as follows: $((x_{j})_{j\in J},\mathcal{I}, x)\in \mathcal{L}_{\mathcal{IS}}$ iff $(x_{j})_{j\in J}\rightarrow _{\mathcal{IS}}x$, where $(x_{j})_{j\in J}$ is a net in $X$, $x\in X$ and $\mathcal{I}$ is an ideal of $J$. Let

\begin{center}
$\mathcal{IS}(X)=\{U\subseteq X:$ for any $((x_{j})_{j\in J}, \mathcal{I}, x)\in \mathcal{L}_{\mathcal{IS}}$, $x\in U$ implies $\{j\in J: x_{j}\notin U\}\in \mathcal{I}\}$.
\end{center}

\begin{prop}\label{5.7}
Let $X$ be a $T_{0}$ topological space, then $(X,\mathcal{IS}(X))$ is a topological space.
\end{prop}

\begin{proof} Since $\emptyset\in \mathcal{I}$, where $\mathcal{I}$ is an ideal of $J$, we have $X\in \mathcal{IS}(X)$. Let $\mathcal{I}$ is an ideal of $J$ and $J\in \mathcal{I}$, then $\emptyset \in \mathcal{IS}(X)$.

If $\{U_{k}\}_{k\in K}$ is a nonempty subfamily of $\mathcal{IS}(X)$ and $(x_{j})_{j\in J}\rightarrow _{\mathcal{IS}}x\in \bigcup_{K\in K}U_{k}$, then there exists $k_{0}\in K$ such that $x\in U_{k}$. So, $\{j\in J: x_{j}\notin U_{k_{0}}\}\in \mathcal{I}$. Since $\{j\in J: x_{j}\notin \bigcup_{k\in K}U_{k}\}\subseteq \{j\in J: x_{j}\notin U_{k_{0}}\}\in \mathcal{I}$, $\{j\in J: x_{j}\notin \bigcup_{k\in K}U_{k}\}\in \mathcal{I}$. Hence, $\bigcup_{k\in K}U_{k}\in \mathcal{IS}(X)$.

If $U_{1}, U_{2}\in \mathcal{IS}(X)$ and $(x_{j})_{j\in J}\rightarrow _{\mathcal{IS}}x\in U_{1}\cap U_{2}$, then $\{j\in J: x_{j}\notin U_{i}\}\in \mathcal{I}$, $i=1,2$. Since $\{j\in J: x_{j}\notin U_{1}\cap U_{2}\}=\{j\in J: x_{j}\notin U_{1}\}\cup\{j\in J: x_{j}\notin U_{2}\}\in \mathcal{I}$, $U_{1}\cap U_{2}\in \mathcal{IS}(L)$.
\end{proof}
\begin{lem}\label{5.8} Let $X$ be a $T_{0}$ topological space. Then $U\in \mathcal{IS}(X)$ iff $U\in \mathcal{D}(X)$.
\end{lem}
\begin{proof} Suppose that $U\in \mathcal{D}(X)$ and $(x_{j})_{j\in J}\rightarrow _{\mathcal{IS}}x\in U$. To show $U\in \mathcal{IS}(X)$, we need only to show that $\{j\in J: x_{j}\notin U\}\in \mathcal{I}$. By Definition \ref{5.3}, there exists a directed subset $D$ of $X$ such that $D\rightarrow_{\tau} x$ and for any $d\in D$, $\{j\in J: x_{j}\ngeq d\}\in \mathcal{I}$. Since $D\rightarrow_{\tau} x$, $d\in D\cap U$ for some $d\in D$. Since $\{j\in J: x_{j}\notin U\}\subseteq \{j\in J: x_{j}\ngeq d\}\in \mathcal{I}$,  $\{j\in J: x_{j}\notin U\}\in \mathcal{I}$. Hence $U\in \mathcal{IS}(X)$.

 On the other hand, let $U\in \mathcal{IS}(X)$ and $D$ be a directed subset of $X$ with $D\rightarrow_{\tau} x\in U$. Consider the net $(x_{d})_{d\in D}$, where $x_{d}=d$, for each $d\in D$. Then $(x_{d})_{d\in D}\rightarrow_{\mathcal{I}_{0}\mathcal{S}}x$ and therefore, by hypothesis, $\{d\in D: x_{d}\notin U\}\in \mathcal{I}_{0}$. Since $\mathcal{I}_{0}$ is non-trivial, $\{d\in D: x_{d}\ngeq x\}\neq D$. Hence, there exists $x_{d}=d\in D$ such that $d\in U$. Thus, $U\in d(X)$. $\Box$
\end{proof}
\begin{cor}\label{5.9} Let $X$ be a monotone determined space, then $\tau=d(X)=\mathcal{IS}(X)$.
\end{cor}
\begin{cor}\label{5.10}
Let $X$ be a monotone determined space, then $\{\Ua_{d}x:x\in X\}$ is a basis of $\mathcal{IS}(X)$.
\end{cor}
\begin{thm}\label{5.11}
 Let $X$ be a monotone determined space, then $X$ is c-space if and only if $(x_{j})_{j\in J}\rightarrow _{\mathcal{IS}}x\Leftrightarrow (x_{j})_{j\in J}\rightarrow _{\mathcal{I}}x$ with respect to $\tau$.
\end{thm}
\begin{proof} Let $x$ be a c-space, then $X$ is monotone determined space. If $(x_{j})_{j\in J}\rightarrow _{\mathcal{IS}}x$, $x\in U\in\tau$, by Corollary \ref{5.9}, $U\in \mathcal{IS}(X)$,  then $\{j\in J: x_{j}\notin U\}\in \mathcal{I}$. Hence, $(x_{j})_{j\in J}\rightarrow _{\mathcal{I}}x$ with respect to $\tau$.

Conversely, suppose that $(x_{j})_{j\in J}\rightarrow _{\mathcal{I}}x$ with respect to $\tau$.
If $J\in \mathcal{I}$, let $D=\{x\}$, then $D\rightarrow_{\tau} x$ and $\{j\in J: x_{j}\notin \uparrow x\}\in \mathcal{I}$. Hence $(x_{j})_{j\in J}\rightarrow _{\mathcal{IS}}x$. Let $\mathcal{I}$ be a non-trivial ideal of $J$ and $(x_{j})_{j\in J}\rightarrow _{\mathcal{I}}x$ with respect to $\tau$. Since $X$ is c-space, $\Da_{d}x$ is directed set and $\Da_{d}x\rightarrow_{\tau} x$. Let $a\ll_{d} x$, then $x\in \Ua_{d} a\in \mathcal{IS}(X)$ by Corollary \ref{5.10}. Since $(x_{j})_{j\in J}\rightarrow _{\mathcal{IS}}x$ and $\Ua_{d} a\in \mathcal{IS}(X)$, $\{j\in J: x_{j}\notin \Ua_{d} a\}\in \mathcal{I}$. Hence $\{j\in J: x_{j}\notin \uparrow a\}\subseteq\{j\in J: x_{j}\notin \Ua_{d} a\}\in \mathcal{I}$, that is, $\{j\in J: x_{j}\notin \uparrow a\}\in \mathcal{I}$. Thus $(x_{j})_{j\in J}\rightarrow _{\mathcal{IS}}x$.

Let $(x_{j})_{j\in J}\rightarrow _{\mathcal{IS}}x\Leftrightarrow (x_{j})_{j\in J}\rightarrow _{\mathcal{I}}x$ with respect to $\tau$. For any $x\in X$, we consider the set of all the open neighborhoods $\mathcal{N}(x)=\{U: x\in U\in \tau\}$ of $x$ and let
\begin{center}
$I=\{(U, n, a)\in \mathcal{N}(x)\times N\times X:a\in U\}$.
\end{center}
 Define an order on $I$: $(U, m, a)\leq(V, n, b)$ iff $V$ is a proper subset of $U$ or $U = V$ and $m\leq n$. Obviously, $I$ is directed set. For each $i=(U, n, a)$, let $x_{i}= a$. Then $(x_{i})_{i\in I}\rightarrow_{\tau} x$. Thus, $(x_{i})_{i\in I}\rightarrow_{\mathcal{I}_{0}} x$, i.e. $(x_{i})_{i\in I}\rightarrow_{\mathcal{I}_{0}\mathcal{S}} x$. By Definition \ref{5.3}, there exists a directed subset $D$ of $X$ such that $D\rightarrow_{\tau} x$ and for any $d\in D$, $\{j\in J: x_{j}\ngeq d\}\in \mathcal{I}_{0}$. By Lemma \ref{2.15}, for any $d\in D$, there exists $i_{d} = (U_{d}, n_{d}, a_{d})\in I$ such that $x_{i}\geq x_{i_{d}}$ for any $i\geq i_{d}$. For any $w\in U_{d}$, $i=(U_{d}, n_{d}, w)\geq (U_{d}, n_{d}, a_{d})=i_{d}$, then $w\geq a_{d}=x_{i_{d}}$ and $w\in U_{d}$. So, there exists $U_{d}$ such that $U_{d}\subseteq \uparrow d$ for any $d\in D$.

For any $d\in D$, let $D'\subseteq X$ be a directed set and $D'\rightarrow_{\tau} x$. Since $x\in U_{d}\in \tau=d(x)$, $U_{d}\cap D'\neq \emptyset$. Let $d'\in U_{d}\cap D'\subseteq \uparrow d$, then $d\leq d'$. Hence $d\ll_{d} x$. Since $D\subseteq \Da_{d} x$, $D\rightarrow_{\tau} x$ and $X$ is monotone determined space, $X$ is c-space.
\end{proof}
\begin{cor}\label{5.12} Let $X$ be a monotone determined space. $\mathcal{L}_{\mathcal{IS}}$ is topological iff $X$ is a d-continuous space iff $X$ is a c-space.
\end{cor}

In the following, we investigate the relationships between $\mathcal{ISL}$-converge and Lawson topology.
\begin{defn}\label{5.13} Let $X$ be a $T_{0}$ topological space. A net $(x_{j})_{j\in J}\subseteq X$ is said to $\mathcal{ISL}$-converge
to a point $x\in X$, where $\mathcal{I}$ is an ideal of $J$, if $(x_{j})_{j\in J}\rightarrow _{\mathcal{IS}}x$ and for any $\{j\in J: x_{j}\ngeq y\}\in \mathcal{I}$, $x\in \uparrow y$. In this case the point $x$ is called the $\mathcal{ISL}$-limit of the net $(x_{j})_{j\in J}$ and we write $(x_{j})_{j\in J}\rightarrow _{\mathcal{ISL}}x$.
\end{defn}
Let
\begin{center}
$\mathcal{ISL}(X)=\{U\subseteq X:$ for any $(x_{j})_{j\in J}\rightarrow _{\mathcal{ISL}}x$, $x\in U$ implies $\{j\in J: x_{j}\notin U\}\in \mathcal{I}\}$.
\end{center}

\begin{prop}\label{5.14} Let $X$ be a $T_{0}$ topological space. Then $\mathcal{IS}(X)\subseteq \mathcal{ISL}(X)$.
\end{prop}
\begin{proof} the proof is direction by Definition \ref{5.13}.
\end{proof}
Let $X$ be a monotone determined space. The Lawson topology on $X$, denoted by $\lambda(X)$, is the topology
generated by $\tau$ and $\omega(X)$, where $\omega(X)$ is the lower topology on $X$ relative to the specialization order
of $X$.

\begin{prop}\label{5.15} Let $X$ be a monotone determined space. Then $\lambda(X)\subseteq \mathcal{ISL}(X)$.
\end{prop}
\begin{proof} We need only to prove that any $U\in \tau\cup \omega(X)$ belongs to $\mathcal{ISL}(X)$.

(1) Suppose $U\in \tau$ and $(x_{j})_{j\in J}\rightarrow _{\mathcal{ISL}}x$ with $x\in U$. Then there exists a directed set $D$
such that $D\rightarrow_{\tau}x$ and for each $d\in D$, $\{j\in J: x_{j}\ngeq d\}\in \mathcal{I}$. So, $D\cap U\neq \emptyset$. Let $d\in D\cap U$, then $\{j\in J: x_{j}\notin U\}\subseteq \{j\in J: x_{j}\ngeq d\}\in \mathcal{I}$. Hence, $U\in\mathcal{ISL}(X)$.

(2) Consider $V=X\backslash \uparrow y$ for any $y\in X$. Given any net $(x_{j})_{j\in J}\rightarrow _{\mathcal{ISL}}x$ and $(x_{j})_{j\in J}\subseteq \uparrow y$, then $\{j\in J: x_{j}\ngeq y\}=\emptyset\in \mathcal{I}$ and $y\leq x$. Hence $x\in \uparrow y$, $V=X\backslash\uparrow y\in \mathcal{ISL}(X)$.

By (1) and (2), we conclude that $\lambda(X)\subseteq \mathcal{ISL}(X)$.
\end{proof}

\begin{thm}\label{5.16} Let $X$ be a c-space, $(x_{j})_{j\in J}\subseteq X$ be a net and $\mathcal{I}$ be a non-trivial ideal of $J$. Then $(x_{j})_{j\in J}\rightarrow _{\mathcal{ISL}}x\Leftrightarrow (x_{j})_{j\in J}\rightarrow _{\mathcal{I}}x$ with respect to $\lambda(X)$.
\end{thm}
\begin{proof} Let $x$ be a c-space. By Proposition \ref{5.15}, $\mathcal{ISL}(X)$ is finer than
$\lambda(X)$. Hence, for any $(x_{j})_{j\in J}\rightarrow _{\mathcal{ISL}}x$ , $(x_{j})_{j\in J}\rightarrow _{\mathcal{I}}x$ with respect to $\lambda(X)$.

Conversely, suppose that $(x_{j})_{j\in J}\rightarrow _{\mathcal{I}}x$ with respect to $\lambda(X)$. Let $\mathcal{I}$ be a non-trivial ideal of $J$ and $\{j\in J: x_{j}\ngeq y\}\in \mathcal{I}$. If $y\nleq x$, then $X\backslash \uparrow y\in \lambda(X)$ and $x\in X\backslash \uparrow y\in \lambda(X)$. So, $\{j\in J: x_{j}\notin X\backslash \uparrow y\}\in \mathcal{I}$. Then $J=\{j\in J: x_{j}\ngeq y\}\cup \{j\in J: x_{j}\notin X\backslash \uparrow y\}\in \mathcal{I}$. Contradiction with $\mathcal{I}$ be a non-trivial ideal of $J$, thus $y\leq x$.

Since $X$ is c-space, $\Da_{d}x$ is directed set and $\Da_{d}x\rightarrow_{\tau} x$. Let $a\ll_{d} x$, then $x\in \Ua_{d} a\in \mathcal{IS}(X)$ by Corollary \ref{5.10}. Since $(x_{j})_{j\in J}\rightarrow _{\mathcal{I}}x$ with respect to $\lambda(X)$ and $\Ua_{d} a\in \mathcal{IS}(X)\subseteq\lambda(X)$, $\{j\in J: x_{j}\notin \Ua_{d} a\}\in \mathcal{I}$. Hence $\{j\in J: x_{j}\notin \uparrow a\}\subseteq\{j\in J: x_{j}\notin \Ua_{d} a\}\in \mathcal{I}$ and $(x_{j})_{j\in J}\rightarrow _{\mathcal{IS}}x$.

Therefore, $(x_{j})_{j\in J}\rightarrow _{\mathcal{ISL}}x$ by Definition \ref{5.13}.

\end{proof}

\section{Ideal $\mathcal{GS}$-convergence and locally hypercompact space}\label{sec6}

In this section, we will define the notion of ideal $\mathcal{GS}$-convergence on a topological spaces
$X$. We show that ideal $\mathcal{GS}$-convergence on a monotone determined space is topological iff $X$ is a locally hypercompact space. These theorems for
ideal $\mathcal{S}$-convergence which is defined in last section can be viewed as special case of ideal $\mathcal{GS}$-convergence.

\begin{defn}\label{6.1}
 Let $X$ be a $T_{0}$ topological space. A net $(x_{j})_{j\in J}\subseteq X$ is said to $\mathcal{IGS}$-converge
to a point $x\in X$, where $\mathcal{I}$ is an ideal of $J$, if there exists a directed family $\mathcal{F}\subseteq \mathcal{P}^{w}(X)$  such that $\mathcal{F}\rightarrow_{\tau} x$ and for any $F\in \mathcal{F}$, $\{j\in J: x_{j}\notin \uparrow F\}\in \mathcal{I}$. In this case the point $x$ is called the $\mathcal{IGS}$-limit of the net $(x_{j})_{j\in J}$ and we write $(x_{j})_{j\in J}\rightarrow _{\mathcal{IGS}}x$.
\end{defn}

\begin{prop}\label{6.2}
Let $X$ be a $T_{0}$ topological space, then $(x_{j})_{j\in J}\rightarrow _{\mathcal{IS}}x$ implies $(x_{j})_{j\in J}\rightarrow _{\mathcal{IGS}}x$.
\end{prop}
\begin{proof} Let $(x_{j})_{j\in J}\rightarrow _{\mathcal{IS}}x$, then there exists a directed subset $D$ of $X$ such that $D\rightarrow_{\tau} x$ and for any $d\in D$, $\{j\in J: x_{j}\ngeq d\}\in \mathcal{I}$ by Definition \ref{5.3}. Let $\mathcal{F}=\{\{d\}:d\in D\}$, then $\mathcal{F}$ is directed and $\mathcal{F}\rightarrow_{\tau} x$. For any $F\in \mathcal{F}$,  $\{j\in J: x_{j}\notin \uparrow F\}=\{j\in J: x_{j}\ngeq d\}\in \mathcal{I}$. Hence, $(x_{j})_{j\in J}\rightarrow _{\mathcal{IGS}}x$.
\end{proof}

In general, the inverse of Proposition \ref{6.2} is not true.

\begin{exmp}\label{6.3} Let $X=N\cup \{a,\infty\}$. Define an order on $X$: $\forall x,y\in X$, $x\leq y\Leftrightarrow y=\infty$ or $x,y\in N$, $x\leq y$. Then $(X, \sigma(X))$ is a directed space and $\leq=\leq_{\sigma}$, where $\leq_{\sigma}$ is specialization order with respect to $\sigma(X)$. Obviously, $X$ is quasicontinuous space but is not continuous space. For any $n\in N$, $\{n,a\}\ll_{d} \{a\}$ and $\mathcal{F}\rightarrow_{\sigma(X)} x$, where $\mathcal{F}=\{\{a,n\}: n\in N\}$. Let $x_{2n}=n$, $x_{2n+1}=a$, $n\in N$, then $(x_{n})_{n\in N}$ is a net. Since $\{m\in N: x_{m}\notin \uparrow \{n,a\}\}\in \mathcal{I}_{0}(N)$, $(x_{n})_{n\in N}\rightarrow _{\mathcal{IGS}}a$. But $(x_{j})_{j\in J}\rightarrow _{\mathcal{IS}}x$ is not hold.
\end{exmp}
Similarly with the Lemma \ref{5.4}, if $(x_{j})_{j\in J}\rightarrow _{\mathcal{IGS}}x$, then $(x_{j})_{j\in J}\rightarrow _{\mathcal{I}}x$.

\begin{prop}\label{6.4} Let $X$ be a $T_{0}$ topological space and $x\in X$, $G\in \mathcal{P}^{w}(X)$. If for any net $(x_{j})_{j\in J}$ in $X$ and any non-trivial ideal $\mathcal{I}$ of $J$ such that $(x_{j})_{j\in J}\rightarrow _{\mathcal{IGS}}y$, we have $\{j\in J: x_{j}\notin \uparrow G\}\in \mathcal{I}$, then $x\ll_{d} y$.
\end{prop}
\begin{proof} Suppose that for any net $(x_{j})_{j\in J}$ in $X$ and any non-trivial ideal $\mathcal{I}$ of $J$ such that $(x_{j})_{j\in J}\rightarrow _{\mathcal{IGS}}x$, we have $\{j\in J: x_{j}\notin \uparrow G\}\in \mathcal{I}$. Let $D$ be a directed subset of $X$ such that $D\rightarrow_{\tau} x$. Consider the family $\mathcal{F}=\{\{d\}: d\in D\}$. Then $\mathcal{F}\rightarrow_{\mathcal{I}_{0}\mathcal{GS}}x$ and therefore, by hypothesis, $\{d\in D: d \notin \uparrow G\}\in \mathcal{I}_{0}$. Since $\mathcal{I}_{0}$ is non-trivial, $\{d\in D: d \notin \uparrow G\}\neq D$. Hence, there exists $d\in D$ such that $d\in \uparrow G$. Thus, $G\ll_{d} y$.
\end{proof}

  \begin{prop}\label{6.5}Let $X$ be a locally hypercompact space, $x\in X$, $G\in \mathcal{P}^{w}(X)$ and $(x_{j})_{j\in J}$ in $X$ be a net. If for any $G\ll_{d} y$ and any non-trivial ideal $\mathcal{I}$ of $J$, we have $\{j\in J: x_{j}\notin \uparrow G\}\in \mathcal{I}$, then $(x_{j})_{j\in J}\rightarrow _{\mathcal{IGS}}y$.
\end{prop}
\begin{proof} Let $(x_{j})_{j\in J}$ be a net in $X$ and for any $G\ll_{d} y$, we have $\{j\in J: x_{j}\notin \uparrow G\}\in \mathcal{I}$, where $\mathcal{I}$ is a non-trivial ideal of $J$. Since the family $fin_{d}(x)$ is directed and $fin_{d}(x)\rightarrow_{\tau}x$, $(x_{j})_{j\in J}\rightarrow _{\mathcal{IGS}}y$ by hypothesis and Definition \ref{6.1}.
\end{proof}

We define a convergence class $\mathcal{L}_{\mathcal{IGS}}$ on $X$ as follows: $((x_{j})_{j\in J},\mathcal{I}, x)\in \mathcal{L}_{\mathcal{IGS}}$ iff $(x_{j})_{j\in J}\rightarrow _{\mathcal{IGS}}x$, where $(x_{j})_{j\in J}$ is a net in $X$, $x\in X$ and $\mathcal{I}$ is an ideal of $J$. Let
\begin{center} $\mathcal{IGS}(X)=\{U\subseteq X:$ for any $((x_{j})_{j\in J}, \mathcal{I}, x)\in \mathcal{L}_{\mathcal{IGS}}$, $x\in U$ implies $\{j\in J: x_{j}\notin U\}\in \mathcal{I}\}$.
\end{center}

\begin{lem}\label{6.6}  Let $X$ be a $T_{0}$ topological space. Then $\mathcal{IGS}(X)=\mathcal{IS}(X)=\mathcal{D}(X)$.
\end{lem}
\begin{proof} Since $\mathcal{L}_{\mathcal{IS}}\subseteq \mathcal{L}_{\mathcal{IGS}}$, $\mathcal{IGS}(X)\subseteq \mathcal{IS}(X)$. By Lemma \ref{5.8}, we need only to show that $\mathcal{D}(X)\subseteq \mathcal{IGS}(X)$.

Suppose that $U\in \mathcal{D}(X)$ and $(x_{j})_{j\in J}\rightarrow _{\mathcal{IGS}}x\in U$. To show $U\in \mathcal{IGS}(X)$, we need only to show that $\{j\in J: x_{j}\notin U\}\in \mathcal{I}$. By definition \ref{6.1}, there exists a directed family $\mathcal{F}$ such that $\mathcal{F}\rightarrow_{\tau} x$ and for each $F\in \mathcal{F}$, $\{j\in J: x_{j}\notin \uparrow F\}\in \mathcal{I}$. Since $\mathcal{F}\rightarrow_{\tau} x$, there exists a $F\in \mathcal{F}$ such that $F\subseteq U$. Since $\{j\in J: x_{j}\notin U\}\subseteq \{j\in J: x_{j}\notin \uparrow F\}\in \mathcal{I}$,  $U\in \mathcal{IGS}(X)$. Thus $\mathcal{D}(X)\subseteq \mathcal{IGS}(X)$.
\end{proof}

\begin{cor}\label{6.7}   Let $X$ be a monotone determined space, then $\tau=\mathcal{D}(X)=\mathcal{IS}(X)=\mathcal{IGS}(X)$.
\end{cor}

\begin{cor}\label{6.8}   Let $X$ be a locally hypercompact space, then $\{\Uparrow_{d}F:F\in \mathcal{P}^{w}(X)\}$ is a basis of $\mathcal{IGS}(X)$.
\end{cor}
\begin{thm}\label{6.9} Let $X$ be a monotone determined space, then $X$ is locally hypercompact space if and only if $(x_{j})_{j\in J}\rightarrow _{\mathcal{IGS}}x\Leftrightarrow (x_{j})_{j\in J}\rightarrow _{\mathcal{I}}x$ with respect to $\tau$.
\end{thm}
\begin{proof} Let $x$ be locally hypercompact space. If $(x_{j})_{j\in J}\rightarrow _{\mathcal{IGS}}x$, $x\in U\in \tau$, then $U\in \mathcal{IGS}(X)$ by Corollary \ref{6.7},  $\{j\in J: x_{j}\notin U\}\in \mathcal{I}$. Hence, $(x_{j})_{j\in J}\rightarrow _{\mathcal{I}}x$ with respect to $\tau$.
Conversly, suppose that $(x_{j})_{j\in J}\rightarrow _{\mathcal{I}}x$ with respect to $\tau$. If $J\in \mathcal{I}$, let $\mathcal{F}=\{\{x\}\}$, then $\mathcal{F}\rightarrow_{\tau} x$ and $\{j\in J: x_{j}\notin \uparrow x\}\in \mathcal{I}$. Hence $(x_{j})_{j\in J}\rightarrow _{\mathcal{IGS}}x$. Let $\mathcal{I}$ be a non-trivial ideal of $J$ and $(x_{j})_{j\in J}\rightarrow _{\mathcal{I}}x$ with respect to $\tau$. Since $X$ is locally hypercompact space, $fin_{d}(x)$ is directed family and $fin_{d}(x)\rightarrow_{\tau} x$. Let $F\ll_{d} x$, then $x\in \Uparrow_{d} F\in \mathcal{IGS}(X)$ by Corollary \ref{6.8}. Since $(x_{j})_{j\in J}\rightarrow _{\mathcal{I}}x$ and $\Uparrow_{d} F\in \mathcal{IGS}(X)$, $\{j\in J: x_{j}\notin \Uparrow_{d} F\}\in \mathcal{I}$. Since $\{j\in J: x_{j}\notin \uparrow F\}\subseteq\{j\in J: x_{j}\notin \Uparrow_{d} F\}\in \mathcal{I}$, $\{j\in J: x_{j}\notin \uparrow F\}\in \mathcal{I}$. Thus $(x_{j})_{j\in J}\rightarrow _{\mathcal{IGS}}x$.

Let $(x_{j})_{j\in J}\rightarrow _{\mathcal{IGS}}x\Leftrightarrow (x_{j})_{j\in J}\rightarrow _{\mathcal{I}}x$ with respect to $\tau$. For any $x\in X$, we consider the set of all the open neighborhoods $\mathcal{N}(x)=\{U: x\in U\in \tau\}$ of $x$ and let $I=\{(U, n, a)\in \mathcal{N}(x)\times N\times X:a\in U\}$. Define an order on $I$: $(U, m, a)\leq(V, n, b)$ iff $V$ is a proper subset of $U$ or $U = V$ and $m\leq n$. For each $i=(U, n, a)$, let $x_{i}= a$. Then $(x_{i})_{i\in I}\rightarrow_{\tau} x$. Thus, $(x_{i})_{i\in I}\rightarrow_{\mathcal{I}_{0}} x$, i.e. $(x_{i})_{i\in I}\rightarrow_{\mathcal{I}_{0}\mathcal{GS}} x$. By Definition \ref{6.1}, there exists a directed family $\mathcal{F}$ such that $\mathcal{F}\rightarrow_{\tau} x$ and for any $F\in \mathcal{F}$, $\{j\in J: x_{j}\notin \uparrow F\}\in \mathcal{I}_{0}$. By Lemma \ref{2.15}, for any $F\in\mathcal{F}$, there exists $i_{F} = (U_{F}, n_{F}, a_{F})\in I$ such that $x_{i}\geq x_{i_{F}}$ for any $i\geq i_{F}$. For any $w\in U_{F}$, $i=(U_{F}, n_{F}, w)\geq (U_{F}, n_{F}, a_{F})=i_{F}$, then $w\geq a_{F}=x_{i_{F}}$ and $w\in U_{F}$. So, for any $F\in \mathcal{F}$, there exists $U_{F}$ such that $U_{F}\subseteq \uparrow F$.

For any $F\in \mathcal{F}$, let $D'\subseteq X$ be a directed set and $D'\rightarrow_{\tau} x$. Since $x\in U_{F}\in \tau=d(x)$, $U_{F}\cap D'\neq \emptyset$. Let $d'\in U_{d}\cap D'\subseteq \uparrow F$, then $d'\in \uparrow F$. Hence $F\ll_{d} x$. Since $\mathcal{F}\subseteq fin_{d} (X)$, $\mathcal{F}\rightarrow_{\tau} x$ and $X$ is monotone determined space, $X$ is locally hypercompact space by Theorem \ref{2.11}. $\Box$
\end{proof}
\begin{cor}\label{6.10} Let $X$ be a monotone determined space. $\mathcal{L}_{\mathcal{IGS}}$ is topological iff $X$ is locally hypercompact space iff $X$ is a d-quasicontinuous space.
\end{cor}
\begin{defn}\label{6.11} Let $X$ be a $T_{0}$ topological space. A net $(x_{j})_{j\in J}\subseteq X$ is said to $\mathcal{IGSL}$-converge
to a point $x\in X$, where $\mathcal{I}$ is an ideal of $J$, if $(x_{j})_{j\in J}\rightarrow _{\mathcal{IGS}}x$ and for any $\{j\in J: x_{j}\notin \uparrow F\}\in \mathcal{I}$, $x\in \uparrow F$. In this case the point $x$ is called the $\mathcal{IGSL}$-limit of the net $(x_{j})_{j\in J}$ and we write $(x_{j})_{j\in J}\rightarrow _{\mathcal{IGSL}}x$.
\end{defn}
Let
\begin{center}
$\mathcal{IGSL}(X)=\{U\subseteq X:$ for any $(x_{j})_{j\in J}\rightarrow _{\mathcal{IGSL}}x$, $x\in U$ implies $\{j\in J: x_{j}\notin U\}\in \mathcal{I}\}$.
\end{center}

\begin{prop}\label{6.12} Let $X$ be a monotone determined space. Then $\lambda(X)\subseteq \mathcal{IGSL}(X)$.
\end{prop}
\begin{proof}We need only to prove that any $U\in \tau\cup \omega(X)$ belongs to $\mathcal{IGSL}(X)$.

(1)  Since $(x_{j})_{j\in J}\rightarrow _{\mathcal{IGSL}}x$ implies $(x_{j})_{j\in J}\rightarrow _{\mathcal{IGS}}x$, then $\mathcal{IGS}(X)\subseteq\mathcal{IGSL}(X)$. By Lemma \ref{6.6}, $d(X)\subseteq \mathcal{IGSL}(X)$.

(2) For $V=X\backslash\uparrow y$ for any $y\in X$,  the proof is analogous to that of  Proposition \ref{5.15}.

By (1) and (2), we conclude that $\lambda(X)\subseteq \mathcal{IGSL}(X)$.
\end{proof}

\begin{prop}\label{6.13} Let $X$ be a monotone determined space, then $\mathcal{IGSL}(X)\subseteq \mathcal{ISL}(X)$, that is, $(x_{j})_{j\in J}\rightarrow _{\mathcal{ISL}}x\Rightarrow (x_{j})_{j\in J}\rightarrow _{\mathcal{IGSL}}x$.
\end{prop}
\begin{proof}If $(x_{j})_{j\in J}\rightarrow _{\mathcal{ISL}}x$, then $(x_{j})_{j\in J}\rightarrow _{\mathcal{IS}}x$ and for any $\{j\in J: x_{j}\ngeq y\}\in \mathcal{I}$, $x\in \uparrow y$ by Definition \ref{5.13}. By Proposition \ref{6.2}, $(x_{j})_{j\in J}\rightarrow _{\mathcal{IGS}}x$.

(1)Let $\mathcal{I}$ trivial ideal of $J$, then $\forall y\in X$, $\{j\in J: x_{j}\ngeq y\}\in \mathcal{I}$, $x\in \uparrow y$, that is, $y\leq x$. Thus, if $\{j\in J: x_{j}\notin \uparrow F\}\in \mathcal{I}$, then $x\in \uparrow F$.

(2)If $\mathcal{I}$ is a non-trivial ideal of $J$. Let $\{j\in J: x_{j}\notin \uparrow F\}\in \mathcal{I}$. If $x\notin \uparrow F$, then $x\in X\backslash \uparrow F$, Since $(x_{j})_{j\in J}\rightarrow _{\mathcal{ISL}}x$ and $X\backslash \uparrow F\in \lambda(X)\subseteq \mathcal{IGSL}(X)$, $\{j\in J: x_{j}\notin X\backslash \uparrow F\}\in \mathcal{I}$. Hence $\{j\in J: x_{j}\in \uparrow F\}\in \mathcal{I}$, $J=\{j\in J: x_{j}\in \uparrow F\}\cup\{j\in J: x_{j}\notin \uparrow F\}\in \mathcal{I}$. Contradiction with $\mathcal{I}$ be a non-trivial ideal of $J$, $x\in \uparrow F$.

By (1) and (2), we conclude that $(x_{j})_{j\in J}\rightarrow _{\mathcal{ISL}}x\Rightarrow (x_{j})_{j\in J}\rightarrow _{\mathcal{IGSL}}x$.
\end{proof}

\begin{thm}\label{6.14} Let $X$ be a locally hypercompact space, $(x_{j})_{j\in J}\subseteq X$ be a net and $\mathcal{I}$ be a non-trivial ideal of $J$. Then $(x_{j})_{j\in J}\rightarrow _{\mathcal{IGSL}}x\Leftrightarrow (x_{j})_{j\in J}\rightarrow _{\mathcal{I}}x$ with respect to $\lambda(X)$.
\end{thm}
\begin{proof} Similarly with the case of c-spaces, we need only to show that for any net $(x_{j})_{j\in J}\rightarrow _{\mathcal{I}}x$ with respect to $\lambda(X)$, $(x_{j})_{j\in J}\rightarrow _{\mathcal{IGSL}}x$.

Let $\mathcal{I}$ be a non-trivial ideal of $J$ and $\{j\in J: x_{j}\notin \uparrow F\}\in \mathcal{I}$. If $x\notin \uparrow F$, then $X\backslash \uparrow F\in \lambda(X)$ and $x\in X\backslash \uparrow F$. So, $\{j\in J: x_{j}\notin X\backslash \uparrow F\}\in \mathcal{I}$. Then $J=\{j\in J: x_{j}\notin \uparrow F\}\cup \{j\in J: x_{j}\notin X\backslash \uparrow F\}\in \mathcal{I}$. Contradiction with $\mathcal{I}$ be a non-trivial ideal of $J$, $x\in \uparrow F$.

Since $X$ is locally hypercompact space, $fin_{d}(x)$ is directed family and $fin_{d}(x)\rightarrow_{\tau} x$. Let $F\ll_{d} x$, then $x\in \Uparrow_{d} a\in \mathcal{IGS}(X)$ by Corollary \ref{6.8}. Since $(x_{j})_{j\in J}\rightarrow _{\mathcal{I}}x$ and $\Uparrow_{d} F\in \mathcal{IGS}(X)$, $\{j\in J: x_{j}\notin \Uparrow_{d} F\}\in \mathcal{I}$. Hence $\{j\in J: x_{j}\notin \uparrow F\}\subseteq\{j\in J: x_{j}\notin \Uparrow_{d} F\}\in \mathcal{I}$ and $(x_{j})_{j\in J}\rightarrow _{\mathcal{IGS}}x$. Therefore  $(x_{j})_{j\in J}\rightarrow _{\mathcal{IGSL}}x$.
\end{proof}

\begin{thm}\label{6.15}\cite{21,27} Let $X$ be a monotone determined space. Then $X$ is c-space iff $X$ is d-quasicontinuous and d-meet continuous.
\end{thm}

\begin{thm}\label{6.16} Let $X$ be a d-meet continuous space and $\mathcal{I}$ be a non-trivial ideal of $J$, then the following three conditions
are equivalent to each other.

(1) $X$ is c-space;

(2) $(x_{j})_{j\in J}\rightarrow _{\mathcal{ISL}}x\Leftrightarrow (x_{j})_{j\in J}\rightarrow _{\mathcal{I}}x$ with respect to $\lambda(X)$;

(3) $(x_{j})_{j\in J}\rightarrow _{\mathcal{IGSL}}x\Leftrightarrow (x_{j})_{j\in J}\rightarrow _{\mathcal{I}}x$ with respect to $\lambda(X)$.
\end{thm}

\begin{proof} $(1)\Rightarrow (2)$ Obviously by Theorem \ref{5.16}.

$(2)\Rightarrow (3)$ we need only to show that if $(x_{j})_{j\in J}\rightarrow _{\mathcal{I}}x$ with respect to $\lambda(X)$, then $(x_{j})_{j\in J}\rightarrow _{\mathcal{IGSL}}x$. Let $(x_{j})_{j\in J}\rightarrow _{\mathcal{I}}x$ with respect to $\lambda(X)$, then $(x_{j})_{j\in J}\rightarrow _{\mathcal{ISL}}x$ and $(x_{j})_{j\in J}\rightarrow _{\mathcal{IGSL}}x$ by Proposition \ref{6.13}.

$(3)\Rightarrow (1)$ For any $x\in X$, we consider the set of all the open neighborhoods $\mathcal{N}(x)=\{U: x\in U\in \lambda(X)\}$ of $x$ and let $I=\{(U, n, a)\in \mathcal{N}(x)\times N\times X:a\in U\}$. Define an order on $I$: $(U, m, a)\leq(V, n, b)$ iff $V$ is a proper subset of $U$ or $U = V$ and $m\leq n$. For each $i=(U, n, a)$, let $x_{i}= a$. Then $(x_{i})_{i\in I}\rightarrow_{\lambda(X)} x$, that is,  $(x_{i})_{i\in I}\rightarrow_{\mathcal{I}_{0}} x$ with respect to $\lambda(X)$. Thus $(x_{i})_{i\in I}\rightarrow_{\mathcal{I}_{0}\mathcal{GSL}} x$. By Definition \ref{6.1} and  \ref{6.11}, there exists a directed family $\mathcal{F}$ such that $\mathcal{F}\rightarrow_{\tau} x$ and for each $F\in \mathcal{F}$, $\{j\in J: x_{j}\notin \uparrow F\}\in \mathcal{I}_{0}$. By Lemma \ref{2.15}, there exists $i_{F} = (U_{F}, n_{F}, a_{F})\in I$ such that $x_{i}\geq x_{i_{F}}$ for any $i\geq i_{F}$. For any $w\in U_{F}$, $i=(U_{F}, n_{F}, w)\geq (U_{F}, n_{F}, a_{F})=i_{F}$, then $w\geq a_{F}=x_{i_{F}}$ and $w\in U_{F}$. So, for any $F\in \mathcal{F}$, there exists $U_{F}$ such that $U_{F}\subseteq \uparrow F$.

Since $x\in U_{F}\in \lambda(X)$, there exist $V\in \tau$ and $W\in \mathcal{P}^{w}(X)$ such that $x\in V\backslash \uparrow W\subseteq U_{F}$. Since $X$ is meet d-continuous, $x\in \uparrow (V\backslash \uparrow W)\in \tau$. Then $x\in (\uparrow F)_{\tau}^{\circ}$. For any directed set $D\rightarrow_{\tau} x$, there exists $d\in D\cap(\uparrow F)_{\tau}^{\circ}$, $d\in \uparrow F$. Hence, $F\ll_{d} x$, $X$ is d-quasicontinuous space. Thus $X$ is c-space by Theorem \ref{6.15}.
\end{proof}

\vspace{1cm} \noindent {\bf Acknowledgments} \small
\def\toto#1#2{\centerline{\hbox to0.7cm{#1\hss}
\parbox[t]{13cm}{#2}}\vspace{2pt}}

This work is supported by the NSFY of China (Nos. 11401435). The authors are grateful
to the referees for their valuable comments which led to the improvement of this paper.

\vspace{1cm} \noindent {\bf Conflict of interest} \small
\def\toto#1#2{\centerline{\hbox to0.7cm{#1\hss}
\parbox[t]{13cm}{#2}}\vspace{2pt}}

The authors declare that there is no conflict of interest in this paper.

\vspace{1cm} \noindent {\bf } \small
\def\toto#1#2{\centerline{\hbox to0.7cm{#1\hss}
\parbox[t]{13cm}{#2}}\vspace{2pt}}

\end{document}